\documentclass{llncs}
\usepackage{amsmath,amssymb}
\usepackage[table]{xcolor}
\usepackage{caption}
\usepackage{subcaption}
\usepackage[]{mathtools}
\usepackage{bbm}
\usepackage{mdframed}
\usepackage[sort,compress]{cite}
\usepackage[numbers,sectionbib]{natbib}
\usepackage[many]{tcolorbox}
\usepackage{algorithmicx}
\usepackage{algorithm}
\usepackage[]{algpseudocode}
\algnewcommand\algorithmicinput{\textbf{Input:}}
\algnewcommand\INPUT{\item[\algorithmicinput]}
\usepackage{eucal}
\usepackage{moreverb,url}
\usepackage[bookmarksopen,bookmarksnumbered,colorlinks=true, linkcolor=cyan,
citecolor = blue, urlcolor = cyan,filecolor=black , pagebackref=false,hypertexnames=false, plainpages=false, pdfpagelabels ]{hyperref}
\usepackage{pgfplots}
\usepgfplotslibrary{fillbetween}
\usepackage[perpage]{footmisc}
\usepackage{soul}
\usepackage{wrapfig}

\usepackage{tikz} 
\usepackage{tkz-graph}
\usepackage{tkz-berge}
	\usetikzlibrary{backgrounds,fit,shapes,snakes,arrows,shapes.geometric,positioning}
	\usetikzlibrary{intersections,patterns,shapes.misc}
	\usetikzlibrary{decorations.pathmorphing}
	\tikzstyle{block} = [rectangle, rounded corners, minimum width=3cm, minimum height=1cm,text centered, draw=black, fill=red!30]
	\tikzstyle{new} = [rectangle, rounded corners, minimum width=1cm, minimum
	height=1cm,text centered, draw=black, fill=blue!10!white, dashed]
	\tikzstyle{arrow} = [thick,->,>=stealth]
	\usetikzlibrary{calc, quotes}
\usepackage{fontawesome}
\DeclareFontFamily{OT1}{pzc}{}
\DeclareFontShape{OT1}{pzc}{m}{it}{<-> s * [1.200] pzcmi7t}{}
\DeclareMathAlphabet{\mathpzc}{OT1}{pzc}{m}{it}

\AtBeginDocument{%
  \DeclareMathAlphabet\PazoBB{U}{fplmbb}{m}{n}%
}

\usepackage{dsfont}
\usepackage{microtype}

\usepackage{diagbox}
\usepackage{multirow}
\usepackage{dashbox}

\usepackage{sidecap}
\algrenewcommand\textproc{}%


\let\oldthempfootnote\thempfootnote
\def\thempfootnote{\text{\oldthempfootnote}}
\usepackage{enumitem}

\setlist[itemize]{leftmargin=*}
\setlist[enumerate]{leftmargin=*}
\colorlet{lightgray}{green!4}

\usepackage{hyperref}

\DeclareMathAlphabet\mathbfcal{OMS}{cmsy}{b}{n}

\DeclareMathOperator*{\rank}{rank}

\let\oldsim\sim
\renewcommand{\sim}{\overset{e}{\oldsim}}

\definecolor{kkGreen}{RGB}{201,232,206}
\definecolor{kkRed}{RGB}{255,196,215}
\definecolor{kkBlue}{RGB}{214,226,255}

\newcommand{\Q}{Q}
\newcommand{\norm}[1]{\left\lVert#1\right\rVert}

\DeclareMathOperator{\tr}{tr}

\DeclareMathOperator{\Stiefel}{St}

\DeclareMathOperator{\rgrad}{grad}

\DeclareMathOperator{\proj}{Proj}

\DeclareMathOperator{\BDiag}{BlockDiag}

\makeatletter
\newcommand{\@chapapp}{\relax}%
\makeatother

\usepackage[title,header]{appendix}

\begin{document}

\mainmatter              
\title{Block-Coordinate Minimization for Large SDPs with Block-Diagonal Constraints}
\titlerunning{}  
%
\author{Yulun Tian
\and Kasra Khosoussi
\and Jonathan P.~How}
\authorrunning{Tian et al.} 

\institute{Laboratory for Information and Decision Systems\\Massachusetts
	Institute of Technology\\Cambridge, MA, USA.\\
	\email{\{yulun,kasra,jhow\}@mit.edu}}

\maketitle              

\begin{abstract}
  The so-called Burer-Monteiro method is a well-studied technique for solving
  large-scale semidefinite programs (SDPs) via low-rank factorization. The main idea is to solve rank-restricted, albeit \emph{non}-convex, surrogates instead of the original
  SDP. Recent works have shown that, in an important class of SDPs with an
  elegant geometric structure, one can find globally optimal solutions to the
  SDP by finding rank-deficient second-order critical points of an unconstrained
  Riemannian optimization problem. Hence, in such problems, the Burer-Monteiro approach can
  provide a scalable and reliable alternative to interior-point methods that
  scale poorly.
  Among various Riemannian optimization methods proposed, block-coordinate
  minimization (BCM) is of particular interest due to its simplicity.
  \citet{ErdogduBCM2018} in their recent work proposed BCM for problems over the
  Cartesian product of unit spheres and 
  provided global convergence rate estimates for the algorithm.
  This report extends the BCM algorithm and the global convergence
  rate analysis of
  \citet{ErdogduBCM2018} from problems over the Cartesian product of unit spheres
  to the Cartesian product of Stiefel manifolds. The latter more general setting
  has important applications such as synchronization over the special
  orthogonal (SO) and special Euclidean (SE) groups.
\end{abstract}

\subsection*{Notations and Preliminaries}
For a $dn \times dn$ matrix $X$ formed by blocks of size $d \times d$,
$X_{[i,j]} \in \mathbb{R}^{d \times d}$ refers to its $(i,j)$-th block.
$\norm{\cdot}_*$ is the matrix nuclear norm.
$\langle \cdot, \cdot \rangle$ is the Frobenius inner product.
The Stiefel manifold is defined as $\Stiefel(d,r) \triangleq \{Y \in
  \mathbb{R}^{r \times d}: Y^\top Y = {I}_d\}$ and is equipped with the
  Riemannian metric induced by its $\mathbb{R}^{r \times
  d}$ embedding.
Define the product manifold $\Stiefel(d,r)^n \triangleq \{Y =
\begin{bmatrix}
Y_1 & Y_2 & \hdots & Y_n
\end{bmatrix} \in \mathbb{R}^{r \times dn}
: Y_i \in \Stiefel(d, r)\}$.
Given a function $F: \Stiefel(d,r)^n \to \mathbb{R}$,
$\nabla F$
and $\rgrad F$ are the ambient Euclidean and Riemannian gradient of $F$, respectively.
We have that $\rgrad F(Y) = \proj_Y (\nabla F(Y))$, where
$\proj$ is the orthogonal projection operator onto the tangent space of
the product manifold at $Y \in \Stiefel(d,r)^n$ \cite{absil2009optimization}.

\section{Problem Formulation}
We are interested in solving large-scale SDPs with the following structure.
\begin{problem}[SDP] Let $\Q \in \mathbb{R}^{dn \times dn}$.
\begin{equation}
	\normalfont
	\begin{aligned}
		& \underset{}{\text{minimize}}
		& & \tr(\Q X), \\
		& \text{subject to}
		&&	X_{[i,i]} = {I}_d, \, \forall i \in [n], \\
		  &   &   & X \succeq 0.
	\end{aligned}
	\label{eq:sdp}
\end{equation}
\label{prob:sdp}
\end{problem}

Problem~\ref{prob:sdp} arises frequently as SDP relaxations of
important non-convex or combinatorial optimization problems. A comprehensive list of
applications can be found in \cite{BoumalBlockDiagonal}. In particular, for
$d=1$, notable examples include Max-Cut \cite{GoemansMaxCutSDP}, MAP inference on
graphical models \cite{FrostigMRFSDP}, and community detection
\cite{Bandeira2016}. For $d > 1$, examples include pose-graph optimization
 (synchronization over the special Euclidean group) \cite{Rosen2016Certifiably},
rotation synchronization (synchronization over the special orthogonal group) \cite{WangRotationSync}, phase synchronization \cite{SingerPhaseSync2011},
and spherical embedding \cite{WilsonSphereEmbedding}.
In many of these applications, the problem size (i.e., $d \times n$) can be
quite large. This is the case, e.g., in pose-graph optimization where
$d \times n$ typically well exceeds $10^4$.
Interior-point methods are often impractical for solving such large-scale instances
of Problem~\ref{prob:sdp}. To address this
issue, \citet{Burer2003ANP} propose to impose a low-rank factorization on the
decision variable $X$, and solve the resulting non-convex problem; see
Problem~\ref{prob:riemannian} below. Note that Problem~\ref{prob:sdp} is
equivalent to Problem~\ref{prob:riemannian} after introducing an additional non-convex constraint
$\rank(X) \leq r$.

\begin{problem}[rank-restricted SDP, Riemannian optimization form]
\begin{equation}
	\normalfont
	\begin{aligned}
		& \underset{}{\text{minimize}}
		& & \tr(\Q Y^\top Y), \\
		& \text{subject to}
		&&	Y \in \Stiefel(d,r)^n.
	\end{aligned}
	\label{eq:riemannian}
\end{equation}
\label{prob:riemannian}
\end{problem}

\begin{remark}
  Following \cite{WangBCM2017,ErdogduBCM2018}, we assume without loss of generality that $Q$ is
symmetric and $Q_{[i,i]} = {0}_d, \forall i \in [n]$. If $Q$ is not symmetric, replacing it with
its symmetric part $\frac{1}{2}(Q + Q^\top)$ does not change the
objective value, since $X = Y^\top Y$ is symmetric. In addition, setting each
$Q_{[i,i]} = {0}_d$ only decreases the objective by a constant value
$\tr(Q_{[i,i]})$.
\end{remark}

It has been established that if $Y$ is a rank-deficient second-order critical point of
Problem~\ref{prob:riemannian}, then it is a \emph{global} minimizer of
Problem~\ref{prob:riemannian}, and furthermore $X = Y^\top Y$ will be a globally optimal solution for
Problem~\ref{prob:sdp}; see \cite{boumal2018deterministicbm} and references
therein. This has motivated the use
of Riemiannian optimization methods for finding \emph{local} solutions to the unconstrained Riemannian
optimization problem (Problem~\ref{prob:riemannian}).
Block-coordinate minimization (BCM) methods, among others, have been proposed for
solving Problem~\ref{prob:riemannian} on the Cartesian product of unit spheres
$\Stiefel(1,r)^n$ \cite{ErdogduBCM2018, WangBCM2017,
Javanmard2015}.
This paper closely follows the recent work of \citet{ErdogduBCM2018} and extends
the BCM algorithm and its \emph{global} convergence rate
analysis to cover the more general case of Cartesian product of Stiefel
manifolds $\Stiefel(d,r)^n$ with arbitrary $d \leq
r$.

\section{Algorithm}
In this section, we present the generalized BCM algorithm for the product manifold $\Stiefel(d,r)^n$. Note that in Problem~\ref{prob:riemannian}, 
the overall cost function can be decomposed as $F(Y) = \tr(\Q Y^\top Y) = \sum_{i=1}^n F_i(Y_i)$, where
the contribution of a single variable $Y_i$ is,
\begin{equation}
  F_i(Y_i) =  \langle G_i, Y_i \rangle, \, G_i \triangleq \sum_{j \in [n]
  \setminus i }Y_j\Q_{[j,i]}.
\end{equation}
We make the crucial observation that after fixing all other variables $Y_j$ where $j \neq i$, the problem of minimizing $F_i(Y_i)$ subject to $Y_i \in \Stiefel(d,r)$ admits a closed-form solution.
Let $U_i\Sigma_iV_i^\top$ be the singular value decomposition of $-G_i$. Then the optimal $Y_i$ is given by
$Y_i^\star = U_i {I}_{r \times d}V_i^\top$ \cite[Theorem 2.1]{Lai2014}.
This motivates the following block-coordinate method, outlined in Algorithm~\ref{alg:BCM}, for solving Problem~\ref{prob:riemannian}. Following \cite{ErdogduBCM2018}, we consider two sampling schemes related to the choice of $p_i$:
\begin{itemize}
\item[$\bullet$] Uniform sampling: $p_i = 1/n, \forall i \in [n]$.
\item[$\bullet$] Importance sampling: $p_i = \norm{G_i}_* / \sum_{j=1}^n \norm{G_j}_*, \forall i \in [n]$.
\end{itemize}

\begin{algorithm}[h]
	\caption{\textsc{\small Block-Coordinate Minimization (BCM) for Problem~\ref{prob:riemannian}}}
	\label{alg:BCM}
	\begin{algorithmic}[1]
		\State Initialize $Y_i^0 \in \Stiefel(r,d), \forall i \in [n]$. Compute
		$G_i^0 = \sum_{j \in [n] \setminus i}Y_jQ_{[j,i]}, \forall i \in [n]$.
		\For{$k=0,1,\hdots$}
		   \State Randomly select $i_k = i$ with probability $p_i, \forall i \in [n]$.
		   \State $Y_{i_k}^{k+1} \leftarrow U_{i_k} {I}_{r \times d}V_{i_k}^\top$, where $U_{i_k}\Sigma_{i_k}V_{i_k}^\top$ is the SVD of $-G_{i_k}^k$.
		   \State $G_{i}^{k+1} \leftarrow G_{i}^k - Y_{i_k}^k\Q_{[i_k,i]} + Y_{i_k}^{k+1}\Q_{[i_k,i]}, \forall i \neq i_k$.
	    \EndFor
	\end{algorithmic}
\end{algorithm}

\section{Global Convergence Analysis}
In this section, we extend the global convergence rate analysis provided by
\citet{ErdogduBCM2018}. Specifically, we show that the established global convergence rate estimates for BCM and the associated proof techniques can be generalized from the Cartesian product of spheres $\Stiefel(1,r)^n$ to the Cartesian product of Stiefel manifolds
$\Stiefel(d,r)^n$, for any $d \leq r$.
Interestingly, our results reduce exactly to the corresponding technical statements in \cite{ErdogduBCM2018} after setting $d=1$.

Recall that in Algorithm~\ref{alg:BCM}, each iteration minimizes the
contribution of a single variable block $Y_{i_k}$ to the cost function while keeping the other blocks fixed.
Thus the sequence of iterates generated by BCM will yield nonincreasing cost
values. The following lemma confirms that this is indeed the case, and
furthermore, quantifies the cost decrease in terms of $Y_i$'s and $G_i$'s.

\begin{lemma}
  Define $F(Y) = \tr(\Q Y^\top Y) = \sum_{i=1}^n F_i(Y_i) = \sum_{i=1}^n \langle G_i, Y_i \rangle$. Let $Y^k$ denote the value of $Y$ at the $k$th iteration of Algorithm~\ref{alg:BCM}. Each iteration of BCM yields a descent on the cost function:
\begin{equation}
F(Y^{k+1}) - F(Y^k) = -2\left(\norm{G_{i_k}^k}_* + \langle G_{i_k}^k, Y_{i_k}^k
\rangle\right) \leq 0.
\label{eq:BCM_descent}
\end{equation}
\label{lem:BCM_descent}
\end{lemma}

In addition to the cost decrease, another key quantity we must investigate is the Frobenius norm of the Riemannian gradient $\norm{\rgrad F(Y)}_F$.

\begin{lemma}
For $i \in [n]$, define $A_i \triangleq \frac{1}{2}(Y_i^\top G_i + G_i^\top Y_i)$. Then,
\begin{equation}
\norm{\rgrad F(Y)}^2_F = 4\sum_{i=1}^{n}(\norm{G_i}^2_F - \norm{A_i}^2_F).
\label{eq:rgrad_norm}
\end{equation}
\label{lem:rgrad_norm}
\end{lemma}

We are now ready to give the main theoretical results of this section. 
Theorem~\ref{thm:bcm_unif_sampling} and Corollary~\ref{cor:bcm_imp_sampling} below establish the global sublinear convergence rate of BCM with uniform sampling and importance sampling, respectively. In particular, after sufficient number of iterations, BCM with either sampling strategy will produce a solution with arbitrarily small gradient norm in expectation. In other words, Algorithm~\ref{alg:BCM} is guaranteed to converge to a first-order critical point of Problem~\ref{prob:riemannian} in expectation. We note that these results generalize the global convergence proof and rate estimates given in Theorem~3.2 and Corollary~3.3 in \cite{ErdogduBCM2018}.

\begin{theorem}
Let $F^\star$ be the optimal value of Problem~\ref{prob:riemannian}.
Define $C_1(\Q) \triangleq \max_i \\ \sum_{j \neq i} \|\Q_{[j,i]}\|_*$.
Then, for any $K \geq \lceil 2dnC_1(\Q)(F(Y^0) - F^\star) / \epsilon \rceil$ iterations, BCM with uniform sampling is guaranteed to return a solution $Y^k$, for some $k \in [K-1]$, such that $\mathbb{E} \norm{\rgrad F(Y^k)}^2_F \leq \epsilon$.
\label{thm:bcm_unif_sampling}
\end{theorem}

\begin{corollary}
Let $F^\star$ be the optimal value of Problem~\ref{prob:riemannian}.
Define $C_2(\Q) \triangleq \sum_{i\neq j} \norm{\Q_{[i,j]}}_*$.
Then, for any $K \geq \lceil 2dC_2(\Q)(F(Y^0) - F^\star) / \epsilon \rceil$ iterations, BCM with importance sampling is guaranteed to return a solution $Y^k$, for some $k \in [K-1]$, such that $\mathbb{E} \norm{\rgrad F(Y^k)}^2_F \leq \epsilon$.
\label{cor:bcm_imp_sampling}
\end{corollary}

\begin{remark}
We note that the BCM algorithm and the analysis presented in this section readily extend to Cartesian product of Stiefel manifolds with different number of orthonormal frames, 
i.e., $\Stiefel(d_1,r) \times \Stiefel(d_2,r) \times \hdots \times \Stiefel(d_n,r)$, where $d_i \leq r, \forall i \in [n]$. A similar global sublinear convergence rate can be proved, with small changes in the constants (for example, $d$ will be replaced with $\max_{i \in [n]} d_i$).  
\end{remark}

\clearpage
{
\footnotesize
\bibliographystyle{plainnat}
\bibliography{scanexchange,slam,optimization}

\begin{thebibliography}{18}
\providecommand{\natexlab}[1]{#1}
\providecommand{\url}[1]{\texttt{#1}}
\expandafter\ifx\csname urlstyle\endcsname\relax
  \providecommand{\doi}[1]{doi: #1}\else
  \providecommand{\doi}{doi: \begingroup \urlstyle{rm}\Url}\fi

\bibitem[mat()]{mathoverflow}
An inequality on elementary symmetric polynomial of eigenvalues.
\newblock MathOverflow.
\newblock URL \url{https://mathoverflow.net/q/323617}.

\bibitem[Absil et~al.(2009)Absil, Mahony, and Sepulchre]{absil2009optimization}
P-A Absil, Robert Mahony, and Rodolphe Sepulchre.
\newblock \emph{Optimization algorithms on matrix manifolds}.
\newblock Princeton University Press, 2009.

\bibitem[Bandeira et~al.(2016)Bandeira, Boumal, and Voroninski]{Bandeira2016}
Afonso~S. Bandeira, Nicolas Boumal, and Vladislav Voroninski.
\newblock On the low-rank approach for semidefinite programs arising in
  synchronization and community detection.
\newblock In \emph{29th Annual Conference on Learning Theory}. PMLR, 2016.

\bibitem[Boumal et~al.(2018)Boumal, Voroninski, and
  Bandeira]{boumal2018deterministicbm}
N.~Boumal, V.~Voroninski, and A.S. Bandeira.
\newblock Deterministic guarantees for {B}urer--{M}onteiro factorizations of
  smooth semidefinite programs.
\newblock \emph{arXiv preprint arXiv:1804.02008}, 2018.

\bibitem[Boumal(2015)]{BoumalBlockDiagonal}
Nicolas Boumal.
\newblock A riemannian low-rank method for optimization over semidefinite
  matrices with block-diagonal constraints, 2015.

\bibitem[Burer and Monteiro(2003)]{Burer2003ANP}
Samuel Burer and Renato D.~C. Monteiro.
\newblock A nonlinear programming algorithm for solving semidefinite programs
  via low-rank factorization.
\newblock \emph{Math. Program.}, 95:\penalty0 329--357, 2003.

\bibitem[Erdogdu et~al.(2018)Erdogdu, Ozdaglar, Parrilo, and
  Vanli]{ErdogduBCM2018}
Murat~A. Erdogdu, Asuman Ozdaglar, Pablo~A. Parrilo, and Nuri~Denizcan Vanli.
\newblock Convergence rate of block-coordinate maximization burer-monteiro
  method for solving large sdps, 2018.

\bibitem[Frostig et~al.(2014)Frostig, Wang, Liang, and Manning]{FrostigMRFSDP}
Roy Frostig, Sida Wang, Percy~S Liang, and Christopher~D Manning.
\newblock Simple map inference via low-rank relaxations.
\newblock In \emph{Advances in Neural Information Processing Systems 27}.
  Curran Associates, Inc., 2014.

\bibitem[Goemans and Williamson(1995)]{GoemansMaxCutSDP}
Michel~X. Goemans and David~P. Williamson.
\newblock Improved approximation algorithms for maximum cut and satisfiability
  problems using semidefinite programming.
\newblock \emph{J. ACM}, 1995.

\bibitem[Horn(1986)]{HornMatrixAnalysis}
Roger~A Horn.
\newblock \emph{Topics in Matrix Analysis}.
\newblock Cambridge University Press, New York, NY, USA, 1986.

\bibitem[Javanmard et~al.(2015)Javanmard, Montanari, and
  Ricci-Tersenghi]{Javanmard2015}
Adel Javanmard, Andrea Montanari, and Federico Ricci-Tersenghi.
\newblock Phase transitions in semidefinite relaxations.
\newblock 2015.

\bibitem[Lai and Osher(2014)]{Lai2014}
Rongjie Lai and Stanley Osher.
\newblock A splitting method for orthogonality constrained problems.
\newblock \emph{Journal of Scientific Computing}, 2014.

\bibitem[Rosen et~al.(2016)Rosen, Carlone, Bandeira, and
  Leonard]{Rosen2016Certifiably}
D.M. Rosen, L.~Carlone, A.S. Bandeira, and J.J. Leonard.
\newblock A certifiably correct algorithm for synchronization over the special
  {Euclidean} group.
\newblock In \emph{Intl. Workshop on the Algorithmic Foundations of Robotics
  (WAFR)}, San Francisco, CA, December 2016.

\bibitem[Singer(2011)]{SingerPhaseSync2011}
A.~Singer.
\newblock Angular synchronization by eigenvectors and semidefinite programming.
\newblock \emph{Applied and Computational Harmonic Analysis}, 2011.

\bibitem[Wang and Singer(2013)]{WangRotationSync}
Lanhui Wang and Amit Singer.
\newblock {Exact and stable recovery of rotations for robust synchronization}.
\newblock \emph{Information and Inference: A Journal of the IMA}, 12 2013.

\bibitem[Wang et~al.(2017)Wang, Chang, and Kolter]{WangBCM2017}
Po-Wei Wang, Wei-Cheng Chang, and J.~Zico Kolter.
\newblock The mixing method: low-rank coordinate descent for semidefinite
  programming with diagonal constraints, 2017.

\bibitem[Wilson and Hancock(2010)]{WilsonSphereEmbedding}
Richard~C. Wilson and Edwin~R. Hancock.
\newblock Spherical embedding and classification.
\newblock In \emph{Structural, Syntactic, and Statistical Pattern Recognition},
  2010.

\bibitem[Zhan(2013)]{ZhanMatrixTheory}
Xingzhi Zhan.
\newblock \emph{Matrix Theory}.
\newblock American Mathematcial Society, 2013.

\end{thebibliography}
}

\clearpage
\begin{appendices}
	\renewcommand{\thesection}{\appendixname~\Alph{section}}

	\section{Proofs}
	\label{sec:proofs}
	\begin{proof}[Lemma~\ref{lem:BCM_descent}]
	Starting from the definition of $F(Y^{k+1})$,
	\begin{equation}
	\begin{aligned}
		F(Y^{k+1}) &= \sum_{i=1}^{n} \langle Y_i^{k+1}, G_i^{k+1} \rangle \\
		&= \langle Y_{i_k}^{k+1}, G_{i_k}^{k+1} \rangle + \sum_{i \neq i_k} \langle Y_i^{k+1}, G_i^{k+1} \rangle \\
		&= \langle Y_{i_k}^{k+1}, G_{i_k}^{k} \rangle + \sum_{i \neq i_k} \langle Y_i^{k}, G_i^{k+1} \rangle \\
		&= -\norm{G_{i_k}^k}_* + \sum_{i \neq i_k} \langle Y_i^{k}, G_i^{k} - Y^k_{i_k}\Q_{[i_k,i]} + Y^{k+1}_{i_k}\Q_{[i_k,i]} \rangle \\
		&= -\norm{G_{i_k}^k}_* + \sum_{i \neq i_k} \langle Y_i^{k}, G_i^{k} \rangle + \sum_{i \neq i_k} \langle Y^k_{i}, (Y^{k+1}_{i_k} - Y^k_{i_k}) \Q_{[i_k,i]} \rangle \\
		&= -\norm{G_{i_k}^k}_* + \sum_{i \neq i_k} \langle Y_i^{k}, G_i^{k} \rangle + \sum_{i \neq i_k} \langle Y^k_{i}\Q_{[i,i_k]}, Y^{k+1}_{i_k} - Y^k_{i_k}  \rangle \\
		&= -\norm{G_{i_k}^k}_* + \sum_{i \neq i_k} \langle Y_i^{k}, G_i^{k} \rangle + \langle G_{i_k}^k, Y^{k+1}_{i_k} - Y^k_{i_k} \rangle \\
		&= -2\norm{G_{i_k}^k}_* - 2 \langle G_{i_k}^k,Y^k_{i_k} \rangle + \sum_{i} \langle Y_i^{k}, G_i^{k} \rangle \\
		&= -2\norm{G_{i_k}^k}_* - 2 \langle G_{i_k}^k,Y^k_{i_k} \rangle + F(Y^k).
	\end{aligned}
	\end{equation}
	Moreover, using the general von Neumann trace theorem,
	we can upper bound the absolute value of the second term in the above expression:
	\begin{equation}
	\begin{aligned}
		|\langle G_{i_k}^k,Y^k_{i_k} \rangle| &\leq \sum_{i=1}^{d} \sigma_i(G_{i_k}^k) \sigma_i(Y_{i_k}^k) \\
		&= \sum_{i=1}^{d} \sigma_i(G_{i_k}^k)
		= \norm{G_{i_k}^k}_*
	\end{aligned}
	\end{equation}
	This ensures that each iteration of the BCM algorithm yields a descent on the objective value:
	\begin{equation*}
	F(Y^{k+1}) - F(Y^k) = -2 (\norm{G_{i_k}^k}_* + \langle G_{i_k}^k,Y^k_{i_k} \rangle) \leq 0.
	\end{equation*}
	\end{proof}

	\clearpage

	\begin{proof}[Lemma~\ref{lem:rgrad_norm}]
	Starting from the definition of the Riemannian gradient,
	\begin{equation}
	\begin{aligned}
	\rgrad F(Y) &= \proj_Y (\nabla F(Y)) \\
	&= 2 \proj_Y(Y\Q) \\
	&= 2 \bigg \{ Y\Q -  Y \frac{1}{2} \BDiag_d(Y^\top Y Q + \Q Y^\top Y) \bigg \}
	\end{aligned}
	\label{eq:rgrad_def}
	\end{equation}
	Expand the terms in $\frac{1}{2}\BDiag_d(Y^\top Y Q + \Q Y^\top Y)$. The $i$th diagonal block is given by,
	\begin{equation}
	\begin{aligned}
	\frac{1}{2}\BDiag_d(Y^\top Y Q + \Q Y^\top Y)_{[i,i]} & = \frac{1}{2}\sum_{j=1}^{n} (Y_i^\top Y_j \Q_{[j,i]} + Q_{[i,j]}Y_j^\top Y_i) \\
	&= \frac{1}{2} Y_i^\top(\sum_{j \neq i} Y_j \Q_{[j,i]}) + \frac{1}{2} (\sum_{j\neq i} Q_{[i,j]}Y_j^\top)^\top Y_i \\
	&= \frac{1}{2} Y_i^\top G_i + \frac{1}{2} G_i^\top Y_i \\
	&= A_i
	\end{aligned}
	\label{eq:expanded}
	\end{equation}
	Above, we have used the assumption that $\Q_{[i,i]} = {0}_d$.
	Using \eqref{eq:expanded} we can now simplify (\ref{eq:rgrad_def}),
	\begin{equation}
	\begin{aligned}
	\rgrad F(Y) &= 2 Y
	\begin{bmatrix}
	-A_1       & \Q_{[1,2]}  &  \hdots & \Q_{[1,n]}   \\
	\Q_{[2,1]} & -A_2      &  \hdots & \Q_{[2,n]} \\
	\vdots     &           &  \ddots & \vdots     \\
	\Q_{[n,1]} & \Q_{[n,2]}&  \hdots & -A_n
	\end{bmatrix}.
	\end{aligned}
	\end{equation}
	View $\rgrad F(Y) \in \mathbb{R}^{r \times dn}$ as a row block matrix with blocks of size $r \times d$. Then, the $[1,i]$-th block can be expressed as,
	\begin{equation}
	\rgrad F(Y)_{[1,i]} = 2(-Y_iA_i + G_i)
	\end{equation}
	The squared Frobenius norm is given by,
	\begin{equation}
	\begin{aligned}
	\norm{\rgrad F(Y)_{[1,i]}}^2_F &= 4 \tr((-Y_iA_i + G_i)^\top (-Y_iA_i + G_i)) \\
	&= 4 \tr (A_i^\top Y_i^\top Y_i A_i  + G_i^\top G_i - G_i^\top Y_i A_i - A_i^\top Y_i^\top G_i ) \\
	&= 4 (\norm{A_i}^2_F + \norm{G_i}^2_F - \tr((G_i^\top Y_i + Y_i^\top G_i)A_i)) \\
	&= 4 (\norm{A_i}^2_F + \norm{G_i}^2_F - 2\norm{A_i}^2_F) \\
	&= 4 (\norm{G_i}^2_F - \norm{A_i}^2_F).
	\end{aligned}
	\end{equation}
	Finally, the squared Frobenius norm of the entire gradient is simply the sum over the individual blocks,
	\begin{equation}
	\norm{\rgrad F(Y)}^2_F = 4\sum_{i=1}^{n}(\norm{G_i}^2_F - \norm{A_i}^2_F).
	\end{equation}
	\end{proof}

	\clearpage

	\begin{proof}[Theorem~\ref{thm:bcm_unif_sampling}]
	We first note that, since $F$ is a continuous function and $\Stiefel(d,r)^n$ is a compact manifold, $F^\star$ must be bounded from below.
	From Lemma~\ref{lem:BCM_descent},
	\begin{equation}
	\begin{aligned}
	F(Y^k) - F(Y^{k+1}) &= 2 (\norm{G_{i_k}^k}_* + \langle G_{i_k}^k, Y_{i_k}^k \rangle) \\
	&= 2 \frac{\norm{G_{i_k}^k}^2_* + \norm{G_{i_k}^k}_*\langle G_{i_k}^k, Y_{i_k}^k \rangle}{\norm{G_{i_k}^k}_*} \\
	&\geq 2 \frac{\norm{G_{i_k}^k}^2_* - \langle G_{i_k}^k, Y_{i_k}^k \rangle^2}{\norm{G_{i_k}^k}_*} \\
	&= 2 \frac{\norm{G_{i_k}^k}^2_* - \tr(A_{i_k}^k)^2}{\norm{G_{i_k}^k}_*} \\
	&\geq 2 \frac{\norm{G_{i_k}^k}^2_F - \norm{A_{i_k}^k}^2_F}{\norm{G_{i_k}^k}_*}
	\end{aligned}
	\end{equation}
	where the first inequality follows from the general von Neumann trace theorem,
	and the second inequality follows from Lemma~\ref{lem:matrix_norm_inequality}.
	Given $Y^k$, take the expectation over the next iteration of BCM,
	\begin{equation}
	\begin{aligned}
	F(Y^k) - \mathbb{E}_k F(Y^{k+1}) &\geq 2\sum_{i=1}^{n} p_i \frac{\norm{G_{i}^k}^2_F - \norm{A_{i}^k}^2_F}{\norm{G_{i}^k}_*} \\
	&\geq \frac{2}{n \max_i\norm{G_{i}^k}_*} \sum_{i=1}^{n} (\norm{G_{i}^k}^2_F - \norm{A_{i}^k}^2_F) \\
	&= \frac{1}{2n \max_i\norm{G_{i}^k}_*} \norm{\rgrad F(Y^k)}^2_F \\
	&\geq \frac{1}{2dn C_1(\Q)} \norm{\rgrad F(Y^k)}^2_F
	\end{aligned}
	\end{equation}
	where the second equality follows from Lemma~\ref{lem:rgrad_norm}.
	The last inequality holds because each $\norm{G_i^k}_*$ can be upper bounded as,
	\begin{equation}
	\begin{aligned}
	\norm{G_i^k}_*  &= \Big\|\sum_{j \neq i} Y^k_j Q_{[j,i]}\Big\|_* \\
	&\leq \sum_{j \neq i} \norm{Y^k_j Q_{[j,i]}}_* \\
	&\leq \sum_{j \neq i} \norm{Y^k_j}_* \norm{Q_{[j,i]}}_* \\
	&= d \sum_{j \neq i} \norm{Q_{[j,i]}}_* \\
	&\leq d C_1(\Q)
	\end{aligned}
	\end{equation}
	To prove the theorem, suppose that
	$\mathbb{E} \norm{\rgrad F(Y^k)}^2_F > \epsilon, \forall k \in [K-1]$ for some integer $K$.
	Then we have,
	\begin{equation}
	\begin{aligned}
	F(Y^0) - F^\star &\geq F(Y^0) - \mathbb{E}F(Y^K) \\
	&= \sum_{k=0}^{K-1} \mathbb{E}[F(Y^k) - F(Y^{k+1})] \\
	&= \sum_{k=0}^{K-1} \mathbb{E}[F(Y^k) - \mathbb{E}_k F(Y^{k+1})] \\
	&\geq \frac{1}{2dn C_1(\Q)} \sum_{k=0}^{K-1} \mathbb{E} \norm{\rgrad F(Y^k)}^2_F \\
	&> \frac{K\epsilon}{2dn C_1(\Q)}
	\end{aligned}
	\end{equation}
	By contradiction, since $F^\star$ is bounded from below, the algorithm returns a solution with $\norm{\rgrad F(Y^k)}^2_F \leq \epsilon$, for some $k \in [K-1]$, provided that
	\begin{equation}
	K > \frac{2dnC_1(\Q)(F(Y^0) - F^\star)}{\epsilon}
	\end{equation}
	\end{proof}

	\clearpage
	\section{Miscellaneous Lemmas}
	\label{sec:misc_lemmas}

	\begin{lemma} Let $M \in \mathbb{R}^{n \times n}$ have ordered singular values $\sigma_1 \geq \hdots \geq \sigma_n \geq 0$ and eigenvalues $\lambda_1, \hdots , \lambda_n$.
	\begin{enumerate}
	\item[(a)] For any $p > 0$,
	\begin{equation}
	\begin{aligned}
	\sum_{i=1}^n |\lambda_i|^p \leq \sum_{i=1}^n \sigma_i^p
	\end{aligned}
	\label{eq:lambda_sigma_inequality}
	\end{equation}
	\item[(b)]
	\begin{equation}
	\begin{aligned}
	\sum_{1 \leq i < j \leq n} |\lambda_i\lambda_j| \leq \sum_{1 \leq i < j \leq n} \sigma_i\sigma_j
	\end{aligned}
	\label{eq:elementary_polynomial_inequality}
	\end{equation}
	\end{enumerate}
	\label{lem:lambda_sigma_inequalities}
	\end{lemma}
	\begin{proof}
	\
	\begin{enumerate}
	\item[(a)] This is a special case of \cite[Theorem 3.3.13(b)]{HornMatrixAnalysis} with $k = n$.
	\item[(b)] From \cite[Theorem 2.16]{ZhanMatrixTheory}, the compound matrix
	  $\wedge^2 M$ has eigenvalues $\lambda_i\lambda_j$ and singular values
	  $\sigma_i\sigma_j$. Applying (\ref{eq:lambda_sigma_inequality}) on
	  $\wedge^2 M$ with $p = 1$ gives the desired inequality.\footnote{The authors
		thank Darij Grinberg \cite{mathoverflow} for his help with the proof.}
	\end{enumerate}
	\end{proof}

	\begin{lemma}
	Let $G \in \mathbb{R}^{r \times d}, Y \in \Stiefel(d,r)$.
	Let $\sigma_i(\cdot)$, $\lambda_i(\cdot)$ return the $i$th algebraically largest singular value and eigenvalue, respectively. For $i \in [d]$,
	\begin{equation}
	\sigma_i(Y^\top G) \leq \sigma_i(G)
	\label{eq:singular_value_inequality}
	\end{equation}
	\label{lem:singular_value_inequality}
	\end{lemma}
	\begin{proof}
	To show (\ref{eq:singular_value_inequality}),
	we can equivalently show that,
	\begin{equation}
	\lambda_i(G^\top Y Y^\top G) \leq \lambda_i(G^\top G)
	\label{eq:eigenvalue_inequality}
	\end{equation}
	We make use of the min-max characterization of eigenvalues,
	\begin{equation}
	\begin{aligned}
	\lambda_i(G^\top Y Y^\top G) &= \max_{\dim(S) = i} \min_{0 \neq v \in S} \frac{v^\top G^\top Y Y^\top G v}{v^\top v} \\
	&= \max_{\dim(S) = i} \min_{0 \neq v \in S} \frac{v^\top G^\top Y Y^\top G v}{v^\top G^\top G v} \cdot \frac{v^\top G^\top G v}{v^\top v} \\
	&\leq \max_{\dim(S) = i} \min_{0 \neq v \in S} \norm{Y^\top} \frac{v^\top G^\top G v}{v^\top v} \\
	&= \max_{\dim(S) = i} \min_{0 \neq v \in S} \frac{v^\top G^\top G v}{v^\top v} \\
	&= \lambda_i(G^\top G)
	\end{aligned}
	\end{equation}
	where the operator norm satisfies $\norm{Y^\top}=1$.
	\end{proof}

	\begin{lemma}
	Let $G \in \mathbb{R}^{r \times d}, Y \in \Stiefel(d,r)$, and define $A = \frac{1}{2}(G^\top Y + Y^\top G)$.
	\begin{equation}
	\tr(A)^2 - \norm{A}^2_F \leq \norm{G}^2_* - \norm{G}^2_F
	\label{eq:matrix_norm_inequality}
	\end{equation}
	\label{lem:matrix_norm_inequality}
	\end{lemma}
	\begin{proof}
	First, note that
	\begin{equation}
	\begin{aligned}
	\norm{A}^2_F &= \frac{1}{2}\tr(G^\top Y G^\top Y) + \frac{1}{2} \norm{G^\top Y}_F^2 \\
	&= \frac{1}{2} \sum_{i=1}^{d} \lambda_i^2(G^\top Y) + \frac{1}{2} \sum_{i=1}^{d} \sigma_i^2(G^\top Y) \\
	&\geq \tr(G^\top Y G^\top Y)
	\end{aligned}
	\label{eq:A_norm_inequality}
	\end{equation}
	where the inequality holds by invoking (\ref{eq:lambda_sigma_inequality}) with $p=2$. Starting from (\ref{eq:A_norm_inequality}), the left hand side of (\ref{eq:matrix_norm_inequality}) can be upper bounded by,
	\begin{equation}
	\begin{aligned}
	\tr(A)^2 - \norm{A}^2_F &\leq \tr(G^\top Y)^2 - \tr(G^\top Y G^\top Y) \\
	&= 2\sum_{1 \leq i < j \leq d} \lambda_i(G^\top Y)\lambda_j(G^\top Y) \\
	&\leq 2\sum_{1 \leq i < j \leq d} |\lambda_i(G^\top Y)\lambda_j(G^\top Y)| \\
	&\leq 2\sum_{1 \leq i < j \leq d} \sigma_i(G^\top Y)\sigma_j(G^\top Y) \\
	&\leq 2\sum_{1 \leq i < j \leq d} \sigma_i(G)\sigma_j(G) \\
	&= \norm{G}^2_* - \norm{G}^2_F
	\end{aligned}
	\label{eq:LHS_inequality}
	\end{equation}
	Above, the third inequality uses part (b) of
	Lemma~\ref{lem:lambda_sigma_inequalities},
	and the fourth inequality uses Lemma~\ref{lem:singular_value_inequality}.
	\end{proof}

\end{appendices}

\end{document}